\documentclass[12pt,twoside,a4paper]{article}
\usepackage{amssymb,amsmath,amsthm, amscd, mathrsfs,enumerate}

\usepackage{color}

\def\red{\color{red}}

\newtheorem{theorem}{Theorem}[section]
\newtheorem{thm}[theorem]{Theorem}

\newtheorem{lemme}[theorem]{Lemma}

\newtheorem{coro}[theorem]{Corollary}

\newtheorem{definition}[theorem]{Definition}

\theoremstyle{plain}
\newtheorem*{namedthm}{\namedthmname}
\newcounter{namedthm}



\newcommand{\R}{\mathbb{R}}
\newcommand{\C}{\mathbb{C}}
\newcommand{\N}{\mathbb{N}}

\numberwithin{equation}{section}

\usepackage{srcltx,mathrsfs}
\usepackage{url}
\usepackage{authblk}
\usepackage{hyperref}
\hypersetup{
	bookmarks=true,         
	unicode=false,         
	pdftoolbar=true,       
	pdfmenubar=true,       
	pdffitwindow=false,    
	pdfstartview={FitH},   
	colorlinks=true,       
	linkcolor=black,         
	citecolor=black,        
	filecolor=black,      
	urlcolor=black}          

\begin{document}
	\def\K{\mathbb{K}}
	\def\R{\mathbb{R}}
	\def\C{\mathbb{C}}
	\def\Z{\mathbb{Z}}
	\def\Q{\mathbb{Q}}
	\def\D{\mathbb{D}}
	\def\N{\mathbb{N}}
	\def\T{\mathbb{T}}
	\def\P{\mathbb{P}}
	\def\A{\mathscr{A}}
	\def\CC{\mathscr{C}}
	\renewcommand{\theequation}{\thesection.\arabic{equation}}
	\renewenvironment{proof}{{\bfseries Proof:}}{\qed}
	\renewcommand{\thelemme}{\empty{}}
	\newtheorem{cond}{C}
	\newtheorem{lemma}{Lemma}[section]
	\newtheorem{corollary}{Corollary}[section]
	\newtheorem{proposition}{Proposition}[section]
	\newtheorem{notation}{Notation}[section]
	\newtheorem{remark}{Remark}[section]
	\newtheorem{example}{Example}[section]
	\newtheorem{probleme}{Problem}[section]
	\bibliographystyle{plain}

\title{\textbf { Introduction to $\partial\bar{\partial}$-Neumann.}}

\author{ Dieynaba Samb$^{\ast}$ }
\affil{University Assane Seck of Ziguinchor, BP: 523 (Sénégal)}
\affil{d.samb20150580@zig.univ.sn$^{\ast}$}

\author{ Souhaibou Sambou }
\affil{University Gaston Berger of Saint-Louis (Sénégal)}
\affil{souhaibou.sambou@ugb.edu.sn }

\author{ Mamadou Eramane Bodian }
\affil{University Assane Seck of Ziguinchor, BP: 523 (Sénégal)}
\affil{ m.bodian@univ-zig.sn}

\author{ Salomon Sambou }
\affil{University Assane Seck of Ziguinchor, BP: 523 (Sénégal)}
\affil{ssambou@univ-zig.sn }
\date{}


\maketitle
\begin{abstract}
	In this paper, we shall use the $L^2$ existence theorems for $\partial\bar{\partial}$ provide by Guy Laville to etablish for existence theorem for the  $\partial\bar{\partial}$-Neumann operator on any bounded pseudoconvex starred domaine $\Omega$ in $\mathbb{C}^n$  based on the work of D. Spencer on the 	$\bar{\partial}$-Neumann.
	  As a result, we obtain a new method for resolving the 	$\partial\bar{\partial}. $
	
\end{abstract}

\section{Introduction}
The $\bar{\partial}$-Neumann problem was introduced in the fiefies as a means  for holomorphic functions on complex manifold(Cf {\red\cite{0}}).
Since then, the main applications to complexe analysis have centered around the solution of the cauchy-Riemann equation $\bar{\partial}u=f$  which arises from the solution of the $\bar{\partial}$-Neumann problem. \\
This problem was first solved on strictly pseudoconvex domain by Kohn(Cf {\red\cite{A}}), who proved existence and regularity proprieties of the $\bar{\partial}$-Neumann operator $N. $ The operator $\bar{\partial}^\ast N$ then solves the $\bar{\partial}$-equation $\bar{\partial}(\bar{\partial}^\ast N f)=f, $ provided $f$ is $\bar{\partial}$-closed and orthogonal to the space of harmonic forms.\\
 The $\bar{\partial}$-Neumann consists of looking for the inverse of the Laplacian. In fact, the two conditions are the injectivity and the surjectivity of the Laplacian.\\	
S. C. Chen and M. C. Shaw in {\color{red}\cite{1}} have established the following result by considering the Laplacian associated with $\bar{\partial}$ given as follows: $$\bar{\partial}\bar{\partial}^\ast+\bar{\partial}^\ast\bar{\partial}. $$
\begin{theorem}\item
\textnormal{
Let $\Omega$ be a pseudoconvex bounded domain in $\mathbb{C}^n$, $n\geq 2. $ For any $0\leq p \leq n, $ $1\leq q \leq n, $ there exists a bounded operator $$N_{p,q}: L^2_{p,q}(\Omega)\longrightarrow L^2_{p,q}(\Omega)$$ such that
\begin{enumerate}\item
$R(N_{p,q})\subset Dom(\square_{p,q})$
\item
$N_{p,q}\square_{p,q}=\square_{p,q}N_{p,q}=Id$ on $Dom(\square_{p,q})$
\item
For any $f\in L^2_{p,q}(\Omega),$ $$f=\bar{\partial}\bar{\partial}^\ast N_{p,q}f\oplus \bar{\partial}^\ast \bar{\partial} N_{p,q}f$$ 
\item
$\bar{\partial}N_{p,q}=N_{p,q+1}\bar{\partial}$ on $Dom(\bar{\partial})$ with $1\leq q\leq n-1$
\item
$\bar{\partial}^\ast N_{p,q}=N_{p,q-1}\bar{\partial}^\ast $ on $Dom(\bar{\partial})^\ast $ with $2\leq q\leq n$
\item
Let $\delta$ be the diameter of $\Omega$, for $f\in L^2_{p,q}(\Omega), $
$$||N_{p,q}f||\leq \frac{e\delta^2}{q}||f|| \mbox  {  and  }||\bar{\partial}N_{p,q}f||\leq \sqrt{\frac{e\delta^2}{q}}||f ||\mbox {  and  } ||\bar{\partial}^\ast N_{p,q}f||\leq \sqrt{\frac{e\delta^2}{q}}||f||. $$ 
\end{enumerate}
}
\end{theorem}
It's in this same logic that we had the idea of introducing the $\partial\bar{\partial}$-Neumann in order to have a new method for solving the equation $\partial\bar{\partial}u=f$  different from the classical method as in {\color{red}\cite{13}} and in {\color{red}\cite{3}} . So we naturally thought of working with the Laplacian associated with the $\partial\bar{\partial}$ that we constructed by analogy with the one associated with the $\bar{\partial}. $ This laplacian is given as follows: $$\square_{\partial\bar{\partial}}=\partial\bar{\partial}(\partial\bar{\partial})^\ast+(\partial\bar{\partial})^\ast\partial\bar{\partial}. $$ 
We would like to establish the following result in $L_{p,q}^{2}(\Omega). $ On the other hand we do not have the injectivity on all the space $L_{p,q}^{2}(\Omega). $  Therefore, we will work on $L_{p,q}^{2}(\Omega)/ ker(\square_{\partial\bar{\partial}}). $ The result is as follows:
\begin{thm} \label{P}\item
	\textnormal{Let $\Omega$ be a strictly pseudoconvex bounded starred domain of $\mathbb{C}^n$, \\ $n\geq 2$ and $R(\square_{\partial\bar{\partial}})$ is closed.  For $1\leq p\leq n$ et $1\leq q\leq n, $ there exists an operator in the sense of the  Remark \eqref{DD} denoted by $N_{\partial\bar{\partial}}$ defined as  $$N_{\partial\bar{\partial}}:  L_{p,q}^{2}(\Omega)/ker(\square_{\partial\bar{\partial}})\longrightarrow L_{p,q}^{2}(\Omega)/ker(\square_{\partial\bar{\partial}})$$ verifying the following properties:}
	\begin{itemize}\item[(1)]
		$R(N_{\partial\bar{\partial}})\subset Dom(\square_{\partial\bar{\partial}})\cap ker(\square_{\partial\bar{\partial}})^\perp$
		\item[(2)]
		$N_{\partial\bar{\partial}}\square_{\partial\bar{\partial}}=\square_{\partial\bar{\partial}}N_{\partial\bar{\partial}}=Id$ on $R(\square_{\partial\bar{\partial}})$
		\item[(3)]
		For all $f\in L^2_{p,q}(\Omega),$ $$f=f_1\oplus\partial\bar{\partial}(\partial\bar{\partial})^\ast N_{\partial\bar{\partial}}f_2\oplus (\partial\bar{\partial})^\ast \partial\bar{\partial} N_{\partial\bar{\partial}}f_2$$ where $f_1\in ker(\square_{\partial\bar{\partial}})$ and $f_2\in R(\square_{\partial\bar{\partial}})$
		\item[(4)]
		$\partial\bar{\partial}N_{p,q}=N_{p+1,q+1}\partial\bar{\partial}$ on $Dom(\partial\bar{\partial})$ with \\$1\leq p\leq n-1, $ $1\leq q\leq n-1$
		\item[(5)]
		$(\partial\bar{\partial})^\ast N_{p,q}=N_{p-1,q-1}(\partial\bar{\partial})^\ast $ on $Dom(\partial\bar{\partial})^\ast $ with $2\leq p\leq n, $ $2\leq q\leq n$
		\item[(6)]
		For $f\in L^2_{p,q}(\Omega), $
		\textnormal{there exists a constant $k\geq 0$ such that:}
	\end{itemize}
	\textnormal{
		$$||N_{\partial\bar{\partial}}f||\leq k||f|| \mbox  {  and }||\partial\bar{\partial}N_{p,q}f||\leq \sqrt{k}||f ||\mbox {  and  } ||(\partial\bar{\partial})^\ast N_{p,q}f||\leq \sqrt{k}||f||. $$}
\end{thm}
This result is the simplest form to obtain the $\partial\bar{\partial}$-Neumann. However, as we know how to solve the $\partial\bar{\partial}$-equation  with a $d$-closed form, we restricted the work to  $L_{p,q}^{2}(\Omega)\cap\ker(d)$ to allow $ker(\square_{\partial\bar{\partial}})=0. $ In other words guarantee injectivity which is one of the essential conditions to obtain $\partial\bar{\partial}$-Neumann. The result is as follows:
\begin{thm} \item
\textnormal{Let $\Omega$ be a strictly pseudoconvex bounded starred domain of $\mathbb{C}^n$, $n\geq 2. $ For $1\leq p\leq n$ and $1\leq q\leq n, $ there exists an operator $$N_{\partial\bar{\partial}}:  L_{p,q}^{2}(\Omega)\cap\ker(d)\longrightarrow L_{p,q}^{2}(\Omega)\cap\ker(d)$$ verifying the following properties:}
\begin{itemize}\item[(1)]
$R(N_{\partial\bar{\partial}})\subset Dom(\square_{\partial\bar{\partial}}), $
\item[(2)]
$N_{\partial\bar{\partial}}\square_{\partial\bar{\partial}}=\square_{\partial\bar{\partial}}N_{\partial\bar{\partial}}=Id$ on $Dom(\square_{\partial\bar{\partial}})$
\item[(3)]
For any $f\in L^2_{p,q}(\Omega)\cap\ker(d),$ $$f=\partial\bar{\partial}(\partial\bar{\partial})^\ast N_{\partial\bar{\partial}}f\bigoplus (\partial\bar{\partial})^\ast \partial\bar{\partial} N_{\partial\bar{\partial}}f$$
\item[(4)]
$\partial\bar{\partial}N_{p,q}=N_{p+1,q+1}\partial\bar{\partial}$ on $Dom(\partial\bar{\partial})$ with $1\leq p\leq n-1, $ $1\leq q\leq n-1. $
\item[(5)]
$(\partial\bar{\partial})^\ast N_{p,q}=N_{p-1,q-1}(\partial\bar{\partial})^\ast $ on $Dom(\partial\bar{\partial})^\ast $ with $2\leq p\leq n, $ $2\leq q\leq n. $
\item[(6)]
For $f\in L^2_{p,q}(\Omega)\cap\ker(d), $
\textnormal{there exists a constant $k\geq 0$ such that :}
\end{itemize}
\textnormal{
$$||N_{\partial\bar{\partial}}f||\leq k||f||\mbox  {  and  } ||\partial\bar{\partial}N_{p,q}f||\leq \sqrt{k}||f||\mbox  {  and  }||(\partial\bar{\partial})^\ast N_{p,q}f||\leq \sqrt{k}||f||. $$ }
\end{thm}
It is important to point out that this Laplacian is not elliptic. Which prevents you from going from $L^2(\Omega)$ to $L^\infty(\Omega)$ even if we obtain the canonical solution. Being always looking for better results, we thought of working with the Bott-Chern Laplacian. With this Laplacian (non-elliptic and elliptic), we obtain the $\partial\bar{\partial}$-Neumann but we do not obtain the canonical solution. The results are given by Theorem \eqref{H} and \eqref{E}.\\
This work will be divided into four parts. The objective in each part is to introduce if possible the $\partial\bar{\partial}$-Neumann and their properties.
\section{Preliminary General}
In this section, we will define a few concepts that will be useful to us later on.
           \textnormal{Let $\Omega \subset \mathbb{C}^n$ be a bounded domain with a smooth boundary.
              The space $L^2(\Omega)$ is defined as: $$ L^2(\Omega)=\lbrace f: \Omega\rightarrow\mathbb{C}\mid \int_\Omega|f|^2 dm<+\infty \rbrace $$ where $dm$ is the volume element. $L^2(\Omega)$ has the norm  $$||f||_2=\big(\int_\Omega|f|^2 dm\big)^\frac{1}{2}$$ which comes from the scalar product defined as follows: $\forall\;\;f,g\in  L^2(\Omega) $  $$<f,g>=\int_\Omega (f\times \overline{g})dm$$ is an Hilbert space.}
                We define the space  $L^2_{p,q}(\Omega)$ as the space of  $(p,q)$-differential forms with coefficients in  $L^2(\Omega). $
                   We define the operator  $\partial\bar{\partial}$ in the sense of distributions as:
                $$\partial\bar{\partial}: L_{p,q}^{2}(\Omega)\longrightarrow L_{p+1,q+1}^{2}(\Omega)$$ and its domain is given by $$Dom(\partial\bar{\partial})=\lbrace f\in L_{p,q}^{2}(\Omega)\mid \partial\bar{\partial}f\in L_{p+1,q+1}^{2}(\Omega)\rbrace. $$\\
                Its adjoint is given by: $$(\partial\bar{\partial})^\ast:L_{p+1,q+1}^{2}(\Omega)\longrightarrow L_{p,q}^{2}(\Omega). $$
                Its domain is given by: $$Dom(\partial\bar{\partial})^\ast=\lbrace f\in L_{p+1,q+1}^{2}(\Omega)\mid (\partial\bar{\partial})^\ast f\in L_{p,q}^{2}(\Omega)\rbrace. $$
               \textnormal{ We define the laplacian of the $\partial\bar{\partial}$ denoted $\square_{\partial\bar{\partial}}$} by: $$\square_{\partial\bar{\partial}}=\partial\bar{\partial}(\partial\bar{\partial})^\ast+(\partial\bar{\partial})^\ast\partial\bar{\partial}. $$
               \begin{definition}\item
             \textnormal{ Let $ \square_{\partial\bar{\partial}}$ be an operator defined by  $L_{p,q}^{2}(\Omega)\longrightarrow L_{p,q}^{2}(\Omega)$ such that
             	 \\$Dom(\square_{\partial\bar{\partial}})=\lbrace f\in Dom(\partial\bar{\partial})\cap Dom(\partial\bar{\partial})^\ast\mid \partial\bar{\partial}f\in Dom(\partial\bar{\partial})^\ast \\\mbox {   and  }  (\partial\bar{\partial})^\ast f\in Dom(\partial\bar{\partial})\rbrace. $}
               \end{definition}
           
   \section{The inverse of $\square_{\partial\bar{\partial}}=\partial\bar{\partial}(\partial\bar{\partial})^\ast+(\partial\bar{\partial})^\ast\partial\bar{\partial}$ on $L_{p,q}^{2}(\Omega)/ ker(\square_{\partial\bar{\partial}}). $ }
   In this section we work with the Laplacian associated with $\partial\bar{\partial}$ given as $$\partial\bar{\partial}(\partial\bar{\partial})^\ast+(\partial\bar{\partial})^\ast\partial\bar{\partial}. $$ 
\subsection{Some Lemmas}
Before proving the main theorem of this section (Theorem \eqref{P}), we will state by a few useful lemmas.
\begin{lemme}\item\label{DS}
\textnormal{The operator $$\square_{\partial\bar{\partial}}:L_{p,q}^{2}(\Omega)\longrightarrow L_{p,q}^{2}(\Omega) $$ is closed, densely-defined and self-adjoint. }
\end{lemme}
\begin{proof}
Let's show that $\square_{\partial\bar{\partial}}$ is closed.\\
Let $f_n\in Dom(\square_{\partial\bar{\partial}})$ such that $ f_n \longrightarrow f \mbox{  with  } f\in Dom(\square_{\partial\bar{\partial}}). $\\ Let's show $\square_{\partial\bar{\partial}}f_n\longrightarrow \square_{\partial\bar{\partial}}f. $\\
We have:
\begin{align*}
<\square_{\partial\bar{\partial}}f_n, f_n> &= <(\partial\bar{\partial}(\partial\bar{\partial})^\ast+(\partial\bar{\partial})^\ast\partial\bar{\partial})f_n, f_n>\\
& = <(\partial\bar{\partial}(\partial\bar{\partial})^\ast)f_n+((\partial\bar{\partial})^\ast\partial\bar{\partial})f_n, f_n>\\
& = <(\partial\bar{\partial}(\partial\bar{\partial})^\ast)f_n, f_n>+<((\partial\bar{\partial})^\ast\partial\bar{\partial})f_n, f_n>\\
& = <(\partial\bar{\partial})^\ast f_n,(\partial\bar{\partial})^\ast f_n>+<\partial\bar{\partial}f_n, \partial\bar{\partial}f_n>\\
& = ||(\partial\bar{\partial})^\ast f_n||^2+||\partial\bar{\partial}f_n||^2. 
\end{align*}
Now $\partial\bar{\partial}$ is a closed operator because it is continuous in the space of distributions.\\
Indeed we have:\\ $<\partial\bar{\partial}f_n, f>=<f_n, (\partial\bar{\partial})^\ast f>\longrightarrow <f, (\partial\bar{\partial})^\ast f>=<\partial\bar{\partial}f, f>. $\\
Thus, $$<\partial\bar{\partial}f_n, f>\longrightarrow <\partial\bar{\partial}f, f> \forall\;\;f\in Dom(\partial\bar{\partial})^\ast\cap Dom(\partial\bar{\partial}). $$ 
So $$\partial\bar{\partial}f_n\longrightarrow \partial\bar{\partial}f. $$
In the same way, we show that $(\partial\bar{\partial})^\ast$ is a closed operator. \\
So we have $$\partial\bar{\partial}^\ast f_n\longrightarrow\partial\bar{\partial}^\ast f. $$
Therefore $$\partial\bar{\partial}(\partial\bar{\partial})^\ast f_n\longrightarrow \partial\bar{\partial}(\partial\bar{\partial})^\ast f. $$
In the same way $$\partial\bar{\partial}f_n\longrightarrow \partial\bar{\partial}f. $$
So $$(\partial\bar{\partial})^\ast\partial\bar{\partial}f_n\longrightarrow (\partial\bar{\partial})^\ast\partial\bar{\partial}f. $$
In conclusion: $$\partial\bar{\partial}(\partial\bar{\partial})^\ast f_n+(\partial\bar{\partial})^\ast\partial\bar{\partial}f_n\longrightarrow \partial\bar{\partial}(\partial\bar{\partial})^\ast f+(\partial\bar{\partial})^\ast\partial\bar{\partial}f. $$
So $$\square_{\partial\bar{\partial}} f_n\longrightarrow \square_{\partial\bar{\partial}} f. $$
$\square_{\partial\bar{\partial}}$ is therefore closed. \\
Let's show that  $\square_{\partial\bar{\partial}}$ is densely-defined. \\
Let $D^{p,q}(\Omega)$ be the space of smooth $(p,q)$ forms  with compact support.\\
We have the following inclusions:
$$D^{p,q}(\Omega)\subset Dom(\square_{\partial\bar{\partial}})\subset L^2_{p,q}(\Omega). $$
And $$\overline{D^{p,q}(\Omega)}\subset\overline{Dom(\square_{\partial\bar{\partial}})}\subset\overline{L^2_{p,q}(\Omega)}. $$
Now  $L^2_{p,q}(\Omega)=\overline{L^2_{p,q}(\Omega)}$ and $D^{p,q}(\Omega)$ is dense in $L^2_{p,q}(\Omega)$ therefore \\$\overline{D^{p,q}(\Omega)}=L^2_{p,q}(\Omega). $
Thus we have: 

 $$\overline{Dom(\square_{\partial\bar{\partial}}})=L^2_{p,q}(\Omega). $$
Therefore $\square_{\partial\bar{\partial}}$ is densely-defined.\\
Let's show that $\square_{\partial\bar{\partial}}$ is self-adjoint.\\
Let $u, v\in Dom(\square_{\partial\bar{\partial}})$ we have: 
\begin{align*}
<\square_{\partial\bar{\partial}}u, v> &= <(\partial\bar{\partial}(\partial\bar{\partial})^\ast+(\partial\bar{\partial})^\ast\partial\bar{\partial})u, v>\\
& = <(\partial\bar{\partial}(\partial\bar{\partial})^\ast)u, v>+<((\partial\bar{\partial})^\ast\partial\bar{\partial})u, v>\\
& = <(\partial\bar{\partial})^\ast u,(\partial\bar{\partial})^\ast v>+<\partial\bar{\partial}u, \partial\bar{\partial}v>\\
& = < u,(\partial\bar{\partial})^\ast \partial\bar{\partial} v>+<u, \partial\bar{\partial}(\partial\bar{\partial})^\ast v>\\
& = <u, ((\partial\bar{\partial})^\ast \partial\bar{\partial}+\partial\bar{\partial}(\partial\bar{\partial})^\ast)v>\\
& = <u, \square_{\partial\bar{\partial}}v>. 
\end{align*}
Thus $<\square_{\partial\bar{\partial}}u, v>=<u,\square_{\partial\bar{\partial}}v>. $
 Therefore $\square_{\partial\bar{\partial}}$ is self-adjoint.
\end{proof}
\begin{remark}\item
\textnormal{According to the Hodge decomposition in Hilbert spaces, we have: $$L^2_{p,q}(\Omega)=\overline{R(\square_{\partial\bar{\partial}})}\oplus ker(\square_{\partial\bar{\partial}}). $$}
\end{remark}

The Neumann problem in its weakest form is to find a criterion for the closedness of $R(\square_{\partial\bar{\partial}}). $\\
\begin{lemma}\item 
	\textnormal{
	Let $$\square_{\partial\bar{\partial}}: L^2_{p,q}(\Omega)/ ker(\square_{\partial\bar{\partial}})\longrightarrow L^2_{p,q}(\Omega)/ ker(\square_{\partial\bar{\partial}}). $$
	If $R(\square_{\partial\bar{\partial}})$ is closed, $\square_{\partial\bar{\partial}}$ is invertible} on  $L^2_{p,q}(\Omega)/ ker(\square_{\partial\bar{\partial}}). $
\end{lemma}
 \begin{proof}
 Since $R(\square_{\partial\bar{\partial}})$ is closed, $\square_{\partial\bar{\partial}}$	is bounded on $Dom(\square_{\partial\bar{\partial}}). $ By Lemma $4.1.1$ of {\color{red}\cite{1}}, we have the estimate $$||f||_1\leq C||\square_{\partial\bar{\partial}}f||_2$$ for all $f\in Dom(\square_{\partial\bar{\partial}})\cap\overline{R(\square_{\partial\bar{\partial}})}. $ According to Hahn Banach's theorem, we can extend to $ L^2_{p,q}(\Omega)/ ker(\square_{\partial\bar{\partial}}). $ Using the Riesz representation theorem, there exists a solution $u$ such that $\square_{\partial\bar{\partial}}u=f. $
 \end{proof}

So we can definie a linear operator $$N_{\partial\bar{\partial}}:L_{p,q}^2(\Omega)/ ker(\square_{\partial\bar{\partial}})\longrightarrow L_{p,q}^2(\Omega)/ ker(\square_{\partial\bar{\partial}})$$ which is the inverse of $\square_{\partial\bar{\partial}}$ and verifies the following remark.
\begin{remark}\item\label{DD}
\textnormal{For $u\in ker(\square_{\partial\bar{\partial}})$, we can extend $$N_{\partial\bar{\partial}}:L_{p,q}^2(\Omega)/ ker(\square_{\partial\bar{\partial}})\longrightarrow L_{p,q}^2(\Omega)/ ker(\square_{\partial\bar{\partial}})$$ to $N_{\partial\bar{\partial}}u=0. $ }
\end{remark}
\subsection{Main theorem}
Based on the previous results, we can prove the main result of this section.

\begin{thm} \label{P}\item
\textnormal{Let $\Omega$ be a strictly pseudoconvex bounded starred domain of $\mathbb{C}^n$, \\ $n\geq 2$ and $R(\square_{\partial\bar{\partial}})$ is closed.  For $1\leq p\leq n$ et $1\leq q\leq n, $ there exists an operator in the sense of the  Remark \eqref{DD} denoted by $N_{\partial\bar{\partial}}$ defined as  $$N_{\partial\bar{\partial}}:  L_{p,q}^{2}(\Omega)/ker(\square_{\partial\bar{\partial}})\longrightarrow L_{p,q}^{2}(\Omega)/ker(\square_{\partial\bar{\partial}})$$ verifying the following properties:}
\begin{itemize}\item[(1)]
$R(N_{\partial\bar{\partial}})\subset Dom(\square_{\partial\bar{\partial}})\cap ker(\square_{\partial\bar{\partial}})^\perp$
\item[(2)]
$N_{\partial\bar{\partial}}\square_{\partial\bar{\partial}}=\square_{\partial\bar{\partial}}N_{\partial\bar{\partial}}=Id$ on $R(\square_{\partial\bar{\partial}})$
\item[(3)]
For all $f\in L^2_{p,q}(\Omega),$ $$f=f_1\oplus\partial\bar{\partial}(\partial\bar{\partial})^\ast N_{\partial\bar{\partial}}f_2\oplus (\partial\bar{\partial})^\ast \partial\bar{\partial} N_{\partial\bar{\partial}}f_2$$ where $f_1\in ker(\square_{\partial\bar{\partial}})$ and $f_2\in R(\square_{\partial\bar{\partial}})$
\item[(4)]
$\partial\bar{\partial}N_{p,q}=N_{p+1,q+1}\partial\bar{\partial}$ on $Dom(\partial\bar{\partial})$ with \\$1\leq p\leq n-1, $ $1\leq q\leq n-1$
\item[(5)]
$(\partial\bar{\partial})^\ast N_{p,q}=N_{p-1,q-1}(\partial\bar{\partial})^\ast $ on $Dom(\partial\bar{\partial})^\ast $ with $2\leq p\leq n, $ $2\leq q\leq n$
\item[(6)]
For $f\in L^2_{p,q}(\Omega), $
\textnormal{there exists a constant $k\geq 0$ such that:}
\end{itemize}
\textnormal{
$$||N_{\partial\bar{\partial}}f||\leq k||f|| \mbox  {  and }||\partial\bar{\partial}N_{p,q}f||\leq \sqrt{k}||f ||\mbox {  and  } ||(\partial\bar{\partial})^\ast N_{p,q}f||\leq \sqrt{k}||f||. $$}
\end{thm}
\begin{proof}
$$\square_{\partial\bar{\partial}}:Dom(\square_{\partial\bar{\partial}})\subset L_{p,q}^{2}(\Omega)/ker(\square_{\partial\bar{\partial}})\longrightarrow L_{p,q}^{2}(\Omega)/ker(\square_{\partial\bar{\partial}}). $$
By definition, the operator $N_{ \partial\bar{\partial}}$ is the inverse of the operator $\square_{\partial\bar{\partial}}$ with $$N_{\partial\bar{\partial}}: L_{p,q}^{2}(\Omega)/ker(\square_{\partial\bar{\partial}})\longrightarrow Dom(\square_{\partial\bar{\partial}})\subset L_{p,q}^{2}(\Omega)/ker(\square_{\partial\bar{\partial}}). $$ The intersection of $Dom(\square_{\partial\bar{\partial}})$ with $ker(\square_{\partial\bar{\partial}})^\perp$ is due to the Remark \eqref{DD}. 
Therefore the properties $(1)$ and $(2)$  are verified i.e $R(N_{\partial\bar{\partial}})\subset Dom(\square_{\partial\bar{\partial}})\cap ker(\square_{\partial\bar{\partial}})^\perp$ and $N_{\partial\bar{\partial}}\square_{\partial\bar{\partial}}=\square_{\partial\bar{\partial}}N_{\partial\bar{\partial}}=Id$ on $R(\square_{\partial\bar{\partial}}). $\\
Let's show the property $(3)$ i.e for all $f\in L^2_{p,q}(\Omega),$ $$f=f_1\oplus\partial\bar{\partial}(\partial\bar{\partial})^\ast N_{\partial\bar{\partial}}f_2\oplus (\partial\bar{\partial})^\ast \partial\bar{\partial} N_{\partial\bar{\partial}}f_2$$ where $f_1\in ker(\square_{\partial\bar{\partial}})$ and $f_2\in R(\square_{\partial\bar{\partial}}). $\\
Since $R(\square_{\partial\bar{\partial}})$ is closed, the Hodge decomposition gives us the following relationship:
$$L^2_{p,q}(\Omega)=ker(\square_{\partial\bar{\partial}})\oplus R(\square_{\partial\bar{\partial}}) . $$
So $\forall\;\;f\in L^2_{p,q}(\Omega)$, $f$ is written as:
$f=f_1\oplus f_2 $ where $f_1\in ker(\square_{\partial\bar{\partial}})$ et $f_2\in R(\square_{\partial\bar{\partial}}). $ So $f$ can be written as 
\begin{equation}\label{B}
f=f_1\oplus\partial\bar{\partial}(\partial\bar{\partial})^\ast N_{p,q}f_2\oplus (\partial\bar{\partial})^\ast\partial\bar{\partial} N_{p,q}f_2. 
\end{equation}
Let's show the property $(4)$ i.e $\partial\bar{\partial}N_{p,q}=N_{p+1,q+1}\partial\bar{\partial}$ on $Dom(\partial\bar{\partial})$ with $1\leq p\leq n-1, $ $1\leq q\leq n-1. $\\
Let $f\in Dom(\partial\bar{\partial}). $ \\
Computing by $\partial\bar{\partial}$ in the relation \eqref{B}, we have:\\
$$\partial\bar{\partial}f=\partial\bar{\partial}f_1\oplus\partial\bar{\partial}\partial\bar{\partial}(\partial\bar{\partial})^\ast N_{p,q}f_2\oplus \partial\bar{\partial}(\partial\bar{\partial})^\ast\partial\bar{\partial} N_{p,q}f_2. $$
Or $$\partial\bar{\partial}f_1=0$$ and $$\partial\bar{\partial}\partial\bar{\partial}(\partial\bar{\partial})^\ast N_{p,q}f=0. $$
So:
$$\partial\bar{\partial}f=\partial\bar{\partial}(\partial\bar{\partial})^\ast\partial\bar{\partial} N_{p,q}f_2. $$
Thus, 
\begin{align*}
N_{p+1, q+1}\partial\bar{\partial}f&=N_{p+1, q+1}\partial\bar{\partial}(\partial\bar{\partial})^\ast\partial\bar{\partial} N_{p,q}f_2\\
&=N_{p+1, q+1}[\partial\bar{\partial}(\partial\bar{\partial})^\ast+(\partial\bar{\partial})^\ast\partial\bar{\partial}]\partial\bar{\partial} N_{p,q}f_2
\end{align*}
Using the property $(2)$, $$N_{p+1, q+1}[\partial\bar{\partial}(\partial\bar{\partial})^\ast+(\partial\bar{\partial})^\ast\partial\bar{\partial}]=I. $$
Therefore, $$N_{p+1, q+1}\partial\bar{\partial}f=\partial\bar{\partial} N_{p,q}f_2. $$ As $N_{p,q}f_1=0$ to the Remark \eqref{DD}. So $$N_{p+1, q+1}\partial\bar{\partial}f=\partial\bar{\partial} N_{p,q}f_2+\partial\bar{\partial} N_{p,q}f_1. $$ Therefore $$N_{p+1, q+1}\partial\bar{\partial}f=\partial\bar{\partial} N_{p,q}f$$ $\forall\;\;f\in Dom(\partial\bar{\partial}). $\\
This proves the property  $(4)$ i.e $$N_{p+1, q+1}\partial\bar{\partial}=\partial\bar{\partial} N_{p,q}. $$
Let's show the property $(5)$ i.e $(\partial\bar{\partial})^\ast N_{p,q}=N_{p-1,q-1}(\partial\bar{\partial})^\ast $ on $Dom(\partial\bar{\partial})^\ast $ with $2\leq p\leq n, $ $2\leq q\leq n. $\\
Let $f\in Dom(\partial\bar{\partial})^\ast. $ \\
Computing by $(\partial\bar{\partial})^\ast$ in the relation \eqref{B}, we have:\\
$$(\partial\bar{\partial})^\ast f=(\partial\bar{\partial})^\ast f_1\oplus(\partial\bar{\partial})^\ast\partial\bar{\partial}(\partial\bar{\partial})^\ast N_{p,q}f_2\oplus (\partial\bar{\partial})^\ast(\partial\bar{\partial})^\ast\partial\bar{\partial}  N_{p,q}f_2. $$
Since $$(\partial\bar{\partial})^\ast f_1=0$$ and $$(\partial\bar{\partial})^\ast(\partial\bar{\partial})^\ast\partial\bar{\partial} N_{p,q}f=0. $$
So:
$$(\partial\bar{\partial})^\ast f=(\partial\bar{\partial})^\ast(\partial\bar{\partial})(\partial\bar{\partial})^\ast N_{p,q}f. $$
Thus, 
\begin{align*}
N_{p-1, q-1}(\partial\bar{\partial})^\ast f&=N_{p-1, q-1}(\partial\bar{\partial})^\ast\partial\bar{\partial}(\partial\bar{\partial})^\ast N_{p,q}f_2\\
&=N_{p-1, q-1}[\partial\bar{\partial}(\partial\bar{\partial})^\ast+(\partial\bar{\partial})^\ast\partial\bar{\partial}](\partial\bar{\partial} )^\ast N_{p,q}f_2
\end{align*}
Using the property $(2)$ of Theorem \eqref{P}, we have $$N_{p-1, q-1}[\partial\bar{\partial}(\partial\bar{\partial})^\ast+(\partial\bar{\partial})^\ast\partial\bar{\partial}]=I $$
Therefore, $$N_{p-1, q-1}(\partial\bar{\partial})^\ast f=(\partial\bar{\partial})^\ast N_{p,q}f_2. $$ As $N_{p,q}f_1=0$ according to the Remark \eqref{DD}. So $$N_{p-1, q-1}(\partial\bar{\partial})^\ast f=(\partial\bar{\partial})^\ast N_{p,q}f_2+(\partial\bar{\partial})^\ast N_{p,q}f_1. $$ So $$N_{p-1, q-1}(\partial\bar{\partial})^\ast f=(\partial\bar{\partial})^\ast N_{p,q}f$$ $\forall\;\;f\in Dom(\partial\bar{\partial})^\ast. $\\
This proves the property $(5)$ i.e $$N_{p+1, q+1}(\partial\bar{\partial})^\ast=(\partial\bar{\partial} )^\ast N_{p,q}. $$
Finally, let's show the property $(6)$ i.e $$||N_{\partial\bar{\partial}}f||\leq k||f|| \mbox  {  and  }||\partial\bar{\partial}N_{p,q}f||\leq \sqrt{k}||f ||\mbox {  and  } ||(\partial\bar{\partial})^\ast N_{p,q}f||\leq \sqrt{k}||f||. $$
According to the property $(1)$ i.e $R(N_{\partial\bar{\partial}})\subset Dom(\square_{\partial\bar{\partial}})\cap ker(\square_{\partial\bar{\partial}})^\perp. $\\
Since $R(\square_{\partial\bar{\partial}})$ is closed, according to the lemme $4.1.1$ of {\red\cite{1}}, we have the following estimate 
\begin{equation}\label{ddd}
 ||f||\leq k|| \square_{\partial\bar{\partial}}f||. 
\end{equation}
Since $f\in Dom(\square_{\partial\bar{\partial}})$, so taking  $f=N_{p,q}f$ in the relation \eqref{ddd}, we have:
\begin{align*}
||N_{p,q}f||&\leq k||\square_{\partial\bar{\partial}}N_{p,q}f||\\
&\leq k||f||.
\end{align*}
Let's computing $||\partial\bar{\partial}N_{p,q}f||^2+||(\partial\bar{\partial})^\ast N_{p,q}f||^2. $\\
\begin{align*}
||\partial\bar{\partial}N_{p,q}f||^2+||(\partial\bar{\partial})^\ast N_{p,q}f||^2&=<\partial\bar{\partial}N_{p,q}f,\partial\bar{\partial}N_{p,q}f>+<(\partial\bar{\partial})^\ast N_{p,q},(\partial\bar{\partial})^\ast N_{p,q}f>\\
&=<N_{p,q}f, (\partial\bar{\partial})^\ast\partial\bar{\partial}N_{p,q}f>+<N_{p,q}f,\partial\bar{\partial}(\partial\bar{\partial})^\ast N_{p,q}f>\\
&=<N_{p,q}f,(\partial\bar{\partial})^\ast\partial\bar{\partial}N_{p,q}f+\partial\bar{\partial}(\partial\bar{\partial})^\ast N_{p,q}f>\\
&=<N_{p,q}f,((\partial\bar{\partial})^\ast\partial\bar{\partial}+\partial\bar{\partial}(\partial\bar{\partial})^\ast) N_{p,q}f>\\
&=<N_{p,q}f,\square_{\partial\bar{\partial}} N_{p,q}f>\\
&=<N_{p,q}f,f>.
\end{align*}
	So $$||\partial\bar{\partial}N_{p,q}f||^2+||(\partial\bar{\partial})^\ast N_{p,q}f||^2=<N_{p,q}f,f>. $$
Using Cauchy-schawtz inequality, we have $$||\partial\bar{\partial}N_{p,q}f||^2+||(\partial\bar{\partial})^\ast N_{p,q}f||^2\leq ||N_{p,q}f||||f||. $$
	Since $$||N_{p,q}f||\leq k||f||. $$
	So $$||\partial\bar{\partial}N_{p,q}f||^2+||(\partial\bar{\partial})^\ast N_{p,q}f||^2\leq k||f||^2. $$
	Therefore, $$||\partial\bar{\partial}N_{p,q}f||^2\leq k||f||^2 \mbox  {    and   } ||(\partial\bar{\partial})^\ast N_{p,q}f||^2\leq k||f||^2. $$
	Thus, $$||\partial\bar{\partial}N_{p,q}f||\leq \sqrt{k}||f|| \mbox {    and   } ||(\partial\bar{\partial})^\ast N_{p,q}f||\leq \sqrt{k}||f||. $$
	 \end{proof} 
	 \\
As a consequence of the Theorem \eqref{P}, we have:
\begin{coro}\item\label{diez}
\textnormal{Under the same assumptions as the  Theorem \eqref{P}, for any \\ $\alpha\in L^2_{p,q}(\Omega)/ker(\square_{\partial\bar{\partial}})\cap ker(d)$ such that $\alpha=\alpha_1+\alpha_2$ with $\alpha_1\in ker(\square_{\partial\bar{\partial}})$ and $\alpha_2\in R(\square_{\partial\bar{\partial}})$. For $f=\alpha-\alpha_1$ there is a differential form $u\in L^2_{p-1, q-1}(\Omega)$ such that $\partial\bar{\partial}u=f. $} 
\end{coro}
\begin{proof}
According to the property $(3)$ of Theorem \eqref{P},  we have $$\alpha=\alpha_1+\partial\bar{\partial}(\partial\bar{\partial})^\ast N_{p,q}\alpha_2+(\partial\bar{\partial})^\ast\partial\bar{\partial}N_{p,q}\alpha_2. $$
So
$$f=\alpha-\alpha_1=\partial\bar{\partial}(\partial\bar{\partial})^\ast N_{p,q}\alpha_2+(\partial\bar{\partial})^\ast\partial\bar{\partial}N_{p,q}\alpha_2. $$
Since $\alpha$ is $d$-closed, $\alpha_2$ is $d$-closed. Therefore, by using the property $(4)$ of Theorem \eqref{P} $$(\partial\bar{\partial})^\ast\partial\bar{\partial}N_{p,q}\alpha_2=0. $$
Donc $$f=\alpha-\alpha_1=\partial\bar{\partial}(\partial\bar{\partial})^\ast N_{p,q}\alpha_2. $$
Since $N_{p,q}\alpha_1=0,$ then $$f=\alpha-\alpha_1=\partial\bar{\partial}(\partial\bar{\partial})^\ast N_{p,q}\alpha_2+ \partial\bar{\partial}(\partial\bar{\partial})^\ast N_{p,q}\alpha_1. $$
Thus
$$f=\partial\bar{\partial}(\partial\bar{\partial})^\ast N_{p,q}\alpha. $$
Consequently, the canonical solution is given by  $$u=(\partial\bar{\partial})^\ast N_{p,q}\alpha. $$


\end{proof}
\section{The inverse of $\square_{\partial\bar{\partial}}=\partial\bar{\partial}(\partial\bar{\partial})^\ast+(\partial\bar{\partial})^\ast\partial\bar{\partial}$ on $L_{p,q}^{2}(\Omega)\cap ker(d). $}
\textnormal{For $\Omega\subset\mathbb{C}^n$ be a strictly  pseudoconvex  bounded  starred domain, $\square_{\partial\bar{\partial}}$ is not injective in the space $L^2_{p,q}(\Omega). $ Therefore, to achieve our objective, we will restrict our study to $L^2_{p,q}(\Omega)\cap ker(d)$ and $\square_{\partial\bar{\partial}}$ is defined as follows: $$\square_{\partial\bar{\partial}}: L_{p,q}^{2}(\Omega)\cap ker(d)\longrightarrow L_{p,q}^{2}(\Omega)\cap ker(d). $$}
\subsection{Some Lemmas}
\begin{lemme}\item\label{OS}
\textnormal{Let $\Omega$ be a strictly pseudoconvex bounded starred domain of $\mathbb{C}^n. $ Then $R(\square_{\partial\bar{\partial}})\cap ker(d)$ is closed and $ker(\square_{\partial\bar{\partial}})=\lbrace 0 \rbrace. $}
\end{lemme}
\begin{proof}
Let's first show that $ker(\square_{\partial\bar{\partial}})=\lbrace 0 \rbrace. $\\
Show that $ker(\square_{\partial\bar{\partial}})=\lbrace 0 \rbrace$  means that $ker(\square_{\partial\bar{\partial}})= ker(\partial\bar{\partial})\cap ker(\partial\bar{\partial})^\ast=\lbrace 0 \rbrace. $\\
Let $\alpha\in ker(\square_{\partial\bar{\partial}})$, we have $\alpha\in Dom(\partial\bar{\partial})\cap Dom(\partial\bar{\partial})^\ast$ and
\begin{align*}
0=<\alpha, \square_{\partial\bar{\partial}}\alpha> &= <\alpha,(\partial\bar{\partial}(\partial\bar{\partial})^\ast+(\partial\bar{\partial})^\ast\partial\bar{\partial})\alpha>\\
&=<\alpha,\partial\bar{\partial}(\partial\bar{\partial})^\ast\alpha+(\partial\bar{\partial})^\ast\partial\bar{\partial}\alpha>\\
&=<\alpha,\partial\bar{\partial}(\partial\bar{\partial})^\ast\alpha>+<\alpha,(\partial\bar{\partial})^\ast\partial\bar{\partial}\alpha>\\
&=<(\partial\bar{\partial})^\ast\alpha,(\partial\bar{\partial})^\ast\alpha>+<\partial\bar{\partial}\alpha,\partial\bar{\partial}\alpha>\\
&=||(\partial\bar{\partial})^\ast\alpha||^2+||\partial\bar{\partial}\alpha||^2.
\end{align*}
We have $\partial\bar{\partial}\alpha=0$ and $(\partial\bar{\partial})^\ast\alpha=0$ so $\alpha\in ker(\partial\bar{\partial})\cap ker(\partial\bar{\partial})^\ast. $\\
Therefore $ker(\square_{\partial\bar{\partial}})\subset ker(\partial\bar{\partial})\cap ker(\partial\bar{\partial})^\ast. $\\
On the other hand, if $\alpha\in ker(\partial\bar{\partial})\cap ker(\partial\bar{\partial})^\ast$, then $\alpha\in Dom(\square_{\partial\bar{\partial}})$ and $\square_{\partial\bar{\partial}}\alpha=0. $ So $\alpha\in ker(\square_{\partial\bar{\partial}}). $\\
Thus $ker(\square_{\partial\bar{\partial}})\supset ker(\partial\bar{\partial})\cap ker(\partial\bar{\partial})^\ast. $\\
To see that $ker(\partial\bar{\partial})\cap ker(\partial\bar{\partial})^\ast=\lbrace 0 \rbrace,$ we take $\alpha\in  L^2_{p,q}(\Omega)\cap ker(d)\cap ker(\square_{\partial\bar{\partial}})$ so $\alpha$ is $d$-closed and $\alpha\in ker(\partial\bar{\partial})\cap ker(\partial\bar{\partial})^\ast. $\\
 Following the same techniques of Proposition $2$  of  {\color{red}\cite{3}}, there exists $u\in L^2_{p-1, q-1}(\Omega)$ such that  $\alpha=\partial\bar{\partial}u. $\\
 We have:\\
\begin{align*}
0=<(\partial\bar{\partial})^\ast\partial\bar{\partial}u, u>&=<\partial\bar{\partial}u, \partial\bar{\partial}u>\\
&=||\partial\bar{\partial}u||^2.
\end{align*}
Therefore $\alpha=0.$ Thus $ker(\square_{\partial\bar{\partial}})=\lbrace 0\rbrace. $\\
Now let's show that $R(\square_{\partial\bar{\partial}})\cap ker(d)$ is closed.\\
Let $f\in L^2_{p,q}(\Omega)$ such that $df=0. $ By following the same steps as the Proposition $2$ of {\color{red}\cite{3}}, there exists $u\in L^2_{p-1, q-1}(\Omega)$ such that $\partial\bar{\partial}u=f$ and $||u||\leq k||f||. $\\
Therefore $R(\partial\bar{\partial})$ is closed according to Lemma $4.1.1$ of  {\color{red}\cite{1}} and is equal to $ker(\partial\bar{\partial})$ considering the cohomology $\frac{ker(\partial\bar{\partial})}{R(\partial\bar{\partial})}. $\\
It follows from the lemma $4.1.1$ of {\color{red}\cite{1}} that $R(\partial\bar{\partial})^\ast $ is closed and we have the following decomposition:
\begin{align*}
L^2_{p,q}(\Omega)\cap ker(d)&=ker(\partial\bar{\partial})\cap ker(d)\oplus R((\partial\bar{\partial})^\ast)\cap ker(d)\\
&=R(\partial\bar{\partial})\cap ker(d)\oplus R((\partial\bar{\partial})^\ast)\cap ker(d).
\end{align*}
Let $f\in Dom(\partial\bar{\partial})\cap Dom(\partial\bar{\partial})^\ast$, we have $f=f_1\oplus f_2$ with $f_1 \in R(\partial\bar{\partial})$ and $f_2\in R((\partial\bar{\partial})^\ast)$ and $\partial\bar{\partial}f_1=0$ , $(\partial\bar{\partial})^\ast f_2=0. $\\
So $f_1, f_2\in Dom(\partial\bar{\partial})\cap Dom(\partial\bar{\partial})^\ast$ and we have: $\partial\bar{\partial}f=\partial\bar{\partial}f_2$ and also $(\partial\bar{\partial})^\ast f=(\partial\bar{\partial})^\ast f_1. $\\
Since $R(\partial\bar{\partial})$ is closed and $R(\partial\bar{\partial})^\ast$ is closed, we have the following estimates: $$||f_2||^2\leq k_1||\partial\bar{\partial}f_2||^2 $$
$$||f_1||^2\leq k_2||(\partial\bar{\partial})^\ast f_1||^2. $$
We have:
$||f||^2=||f_1||^2+||f_2||^2\leq k_1||\partial\bar{\partial}f_2||^2+ k_2||(\partial\bar{\partial})^\ast f_1||^2. $\\
If we take $k=max(k_1,k_2)$, we have: $$||f||^2\leq k(||\partial\bar{\partial}f||^2+||(\partial\bar{\partial})^\ast f||^2). $$ For each $f\in Dom(\square_{\partial\bar{\partial}})$:
\begin{align*}
||f||^2 &\leq k[ <\partial\bar{\partial}f,\partial\bar{\partial}f>+<(\partial\bar{\partial})^\ast f,(\partial\bar{\partial})^\ast f>]\\
&=k[ <f,(\partial\bar{\partial})^\ast\partial\bar{\partial}f>+< f,\partial\bar{\partial}(\partial\bar{\partial})^\ast f>]\\
&= k[ <f,((\partial\bar{\partial})^\ast\partial\bar{\partial}+\partial\bar{\partial}(\partial\bar{\partial})^\ast) f>]\\
&=k<f, \square_{\partial\bar{\partial}}f>
\end{align*}
Using the Cauchy-Schwartz inequality, we have:
$$||f||^2\leq k ||\square_{\partial\bar{\partial}}f||||f||. $$
Finally, we have the following relation:
\begin{equation}\label{A}
||f||\leq k|| \square_{\partial\bar{\partial}}f||.
\end{equation}
Thus, according to the lemma $4.1.1$ of  {\color{red}\cite{1}}, $R(\square_{\partial\bar{\partial}})\cap ker(d)$ is closed. 
\end{proof}
\begin{remark}\item
\textnormal{From Lemma 2, $ker(\square_{\partial\bar{\partial}})=\lbrace 0 \rbrace \Rightarrow \square_{\partial\bar{\partial}}$ is injective. \\
And according to the Hodge decomposition, $$L^2_{p,q}(\Omega)\cap ker(d)=R(\square_{\partial\bar{\partial}})\cap ker(d)$$ so $\square_{\partial\bar{\partial}}$ is surjective.
Therefore the $\square_{\partial\bar{\partial}}$ is invertible. }
\end{remark}
Now we can prove the main theorem of this section.
\subsection{ Main theorem}
\begin{thm} \label{G}\item
\textnormal{Let $\Omega$ be a strictly pseudoconvex bounded starred domain of $\mathbb{C}^n$, $n\geq 2. $ For $1\leq p\leq n$ et $1\leq q\leq n, $ there is an operator  $$N_{\partial\bar{\partial}}:  L_{p,q}^{2}(\Omega)\cap\ker(d)\longrightarrow L_{p,q}^{2}(\Omega)\cap\ker(d)$$ with the following properties:}
\begin{itemize}\item[(1)]
$R(N_{\partial\bar{\partial}})\subset Dom(\square_{\partial\bar{\partial}}), $
\item[(2)]
$N_{\partial\bar{\partial}}\square_{\partial\bar{\partial}}=\square_{\partial\bar{\partial}}N_{\partial\bar{\partial}}=Id$ on $Dom(\square_{\partial\bar{\partial}})$
\item[(3)]
For all $f\in L^2_{p,q}(\Omega)\cap\ker(d),$ $$f=\partial\bar{\partial}(\partial\bar{\partial})^\ast N_{\partial\bar{\partial}}f\bigoplus (\partial\bar{\partial})^\ast \partial\bar{\partial} N_{\partial\bar{\partial}}f$$
\item[(4)]
$\partial\bar{\partial}N_{p,q}=N_{p+1,q+1}\partial\bar{\partial}$ on $Dom(\partial\bar{\partial})$ with $1\leq p\leq n-1, $ $1\leq q\leq n-1. $
\item[(5)]
$(\partial\bar{\partial})^\ast N_{p,q}=N_{p-1,q-1}(\partial\bar{\partial})^\ast $ on $Dom(\partial\bar{\partial})^\ast $ with $2\leq p\leq n, $ $2\leq q\leq n. $
\item[(6)]
For $f\in L^2_{p,q}(\Omega)\cap\ker(d), $
\textnormal{there is a constant $k\geq0$ such that:}
\end{itemize}
\textnormal{
$$||N_{\partial\bar{\partial}}f||\leq k||f||\mbox  {  and  } ||\partial\bar{\partial}N_{p,q}f||\leq \sqrt{k}||f||\mbox  {  and  }||(\partial\bar{\partial})^\ast N_{p,q}f||\leq \sqrt{k}||f||. $$}
\end{thm}

\begin{proof}
We have:
$$\square_{\partial\bar{\partial}}:Dom(\square_{\partial\bar{\partial}})\subset L^2_{p,q}(\Omega)\cap ker(d)\longrightarrow L^2_{p,q}(\Omega)\cap ker(d). $$
By definition, the operator $N_{ \partial\bar{\partial}}$ is the inverse of the operator $\square_{\partial\bar{\partial}}$ with $$N_{\partial\bar{\partial}}: L^2_{p,q}(\Omega)\cap ker(d)\longrightarrow Dom(\square_{\partial\bar{\partial}})\subset L^2_{p,q}(\Omega)\cap ker(d). $$
Consequently, the properties $(1)$ and $(2)$  are verified i.e $R(N_{\partial\bar{\partial}})\subset Dom(\square_{\partial\bar{\partial}})$ and $N_{\partial\bar{\partial}}\square_{\partial\bar{\partial}}=\square_{\partial\bar{\partial}}N_{\partial\bar{\partial}}=Id$ on $Dom(\square_{\partial\bar{\partial}}). $\\
Let's show the property $(3)$ i.e for all $f\in L^2_{p,q}(\Omega)\cap\ker(d),$ $$f=\partial\bar{\partial}(\partial\bar{\partial})^\ast N_{\partial\bar{\partial}}f\bigoplus (\partial\bar{\partial})^\ast \partial\bar{\partial} N_{\partial\bar{\partial}}f. $$
Since $L^2_{p,q}(\Omega)\cap ker(d)$ is a Hilbert space and $ker(\square_{\partial\bar{\partial}})=\lbrace 0\rbrace $, the Hodge decomposition gives us the following relation:
\begin{align*}
L^2_{p,q}(\Omega)\cap ker(d)&=R(\square_{\partial\bar{\partial}})\cap ker(d)\\
&=(\partial\bar{\partial}(\partial\bar{\partial})^\ast(Dom(\square_{\partial\bar{\partial}}))\oplus (\partial\bar{\partial})^\ast\partial\bar{\partial}(Dom(\square_{\partial\bar{\partial}})))\cap ker(d). 
\end{align*}
So $\forall\;\;f\in L^2_{p,q}(\Omega)\cap ker(d)$, $f$ is written as:
\begin{equation}\label{Z}
f=\partial\bar{\partial}(\partial\bar{\partial})^\ast N_{p,q}f\oplus (\partial\bar{\partial})^\ast\partial\bar{\partial} N_{p,q}f.
\end{equation}
Let's show the property $(4)$ i.e $\partial\bar{\partial}N_{p,q}=N_{p+1,q+1}\partial\bar{\partial}$ on $Dom(\partial\bar{\partial})$ with $1\leq p\leq n-1, $ $1\leq q\leq n-1. $\\
Let $f\in Dom(\partial\bar{\partial}). $ \\
By computing  $\partial\bar{\partial}$ in the relation \eqref{Z}, we have:\\
$$\partial\bar{\partial}f=\partial\bar{\partial}\partial\bar{\partial}(\partial\bar{\partial})^\ast N_{p,q}f\oplus \partial\bar{\partial}(\partial\bar{\partial})^\ast\partial\bar{\partial} N_{p,q}f. $$
Or $$\partial\bar{\partial}\partial\bar{\partial}(\partial\bar{\partial})^\ast N_{p,q}f=0. $$
Donc:
$$\partial\bar{\partial}f=\partial\bar{\partial}(\partial\bar{\partial})^\ast\partial\bar{\partial} N_{p,q}f. $$
Thus, 
\begin{align*}
N_{p+1, q+1}\partial\bar{\partial}f&=N_{p+1, q+1}\partial\bar{\partial}(\partial\bar{\partial})^\ast\partial\bar{\partial} N_{p,q}f\\
&=N_{p+1, q+1}[\partial\bar{\partial}(\partial\bar{\partial})^\ast+(\partial\bar{\partial})^\ast\partial\bar{\partial}]\partial\bar{\partial} N_{p,q}f
\end{align*}
Using the property $(2)$, $$N_{p+1, q+1}[\partial\bar{\partial}(\partial\bar{\partial})^\ast+(\partial\bar{\partial})^\ast\partial\bar{\partial}]=I. $$
Therefore, $$N_{p+1, q+1}\partial\bar{\partial}f=\partial\bar{\partial} N_{p,q}f$$ $\forall\;\;f\in Dom(\partial\bar{\partial}). $\\
This proves the property $(4)$ i.e $$N_{p+1, q+1}\partial\bar{\partial}=\partial\bar{\partial} N_{p,q}. $$
Let's show the property $(5)$ i.e $(\partial\bar{\partial})^\ast N_{p,q}=N_{p-1,q-1}(\partial\bar{\partial})^\ast $ on $Dom(\partial\bar{\partial})^\ast $ with $2\leq p\leq n, $ $2\leq q\leq n. $\\
Let $f\in Dom(\partial\bar{\partial})^\ast. $ \\
By computing $(\partial\bar{\partial})^\ast$ in the relation \eqref{Z}, we have:\\
$$(\partial\bar{\partial})^\ast f=(\partial\bar{\partial})^\ast\partial\bar{\partial}(\partial\bar{\partial})^\ast N_{p,q}f\oplus (\partial\bar{\partial})^\ast(\partial\bar{\partial})^\ast\partial\bar{\partial}  N_{p,q}f. $$
Or $$(\partial\bar{\partial})^\ast(\partial\bar{\partial})^\ast\partial\bar{\partial} N_{p,q}f=0. $$
So:
$$(\partial\bar{\partial})^\ast f=(\partial\bar{\partial})^\ast(\partial\bar{\partial})(\partial\bar{\partial})^\ast N_{p,q}f. $$
Ainsi, 
\begin{align*}
N_{p-1, q-1}(\partial\bar{\partial})^\ast f&=N_{p-1, q-1}(\partial\bar{\partial})^\ast\partial\bar{\partial}(\partial\bar{\partial})^\ast N_{p,q}f\\
&=N_{p-1, q-1}[\partial\bar{\partial}(\partial\bar{\partial})^\ast+(\partial\bar{\partial})^\ast\partial\bar{\partial}](\partial\bar{\partial} )^\ast N_{p,q}f
\end{align*}
Using the property $(2)$, $$N_{p-1, q-1}[\partial\bar{\partial}(\partial\bar{\partial})^\ast+(\partial\bar{\partial})^\ast\partial\bar{\partial}]=I $$
Therefore, $$N_{p-1, q-1}(\partial\bar{\partial})^\ast f=(\partial\bar{\partial})^\ast N_{p,q}f$$ $\forall\;\;f\in Dom(\partial\bar{\partial}). $\\
This proves the property $(5)$ i.e $$N_{p+1, q+1}(\partial\bar{\partial})^\ast=(\partial\bar{\partial} )^\ast N_{p,q}. $$
Finally, let's show the property $(6)$ i.e $$||N_{\partial\bar{\partial}}f||\leq k||f||\mbox  {  and  } ||\partial\bar{\partial}N_{p,q}f||\leq \sqrt{k}||f||\mbox  {  and  }||(\partial\bar{\partial})^\ast N_{p,q}f||\leq \sqrt{k}||f||. $$
According to the property $(1)$ i.e $R(N_{\partial\bar{\partial}})\subset Dom(\square_{\partial\bar{\partial}}). $\\
Since $f\in Dom(\square_{\partial\bar{\partial}})$, we have:
\begin{align*}
||N_{p,q}f||&\leq k||\square_{\partial\bar{\partial}}N_{p,q}f||\\
&\leq k||f||.
\end{align*}
Let's computing $||\partial\bar{\partial}N_{p,q}f||^2+||(\partial\bar{\partial})^\ast N_{p,q}f||^2. $\\
\begin{align*}
||\partial\bar{\partial}N_{p,q}f||^2+||(\partial\bar{\partial})^\ast N_{p,q}f||^2&=<\partial\bar{\partial}N_{p,q}f,\partial\bar{\partial}N_{p,q}f>+<(\partial\bar{\partial})^\ast N_{p,q},(\partial\bar{\partial})^\ast N_{p,q}f>\\
&=<N_{p,q}f, (\partial\bar{\partial})^\ast\partial\bar{\partial}N_{p,q}f>+<N_{p,q}f,\partial\bar{\partial}(\partial\bar{\partial})^\ast N_{p,q}f>\\
&=<N_{p,q}f,(\partial\bar{\partial})^\ast\partial\bar{\partial}N_{p,q}f+\partial\bar{\partial}(\partial\bar{\partial})^\ast N_{p,q}f>\\
&=<N_{p,q}f,((\partial\bar{\partial})^\ast\partial\bar{\partial}+\partial\bar{\partial}(\partial\bar{\partial})^\ast) N_{p,q}f>\\
&=<N_{p,q}f,\square_{\partial\bar{\partial}} N_{p,q}f>\\
&=<N_{p,q}f,f>.
\end{align*}
	So $$||\partial\bar{\partial}N_{p,q}f||^2+||(\partial\bar{\partial})^\ast N_{p,q}f||^2=<N_{p,q}f,f>. $$
	Using the Cauchy-schawtz inequality, we have $$||\partial\bar{\partial}N_{p,q}f||^2+||(\partial\bar{\partial})^\ast N_{p,q}f||^2\leq ||N_{p,q}f||||f||. $$
	Since $$||N_{p,q}f||\leq k||f||. $$
	So $$||\partial\bar{\partial}N_{p,q}f||^2+||(\partial\bar{\partial})^\ast N_{p,q}f||^2\leq k||f||^2. $$
	Therefore, $$||\partial\bar{\partial}N_{p,q}f||^2\leq k||f||^2 \mbox  {  and } |(\partial\bar{\partial})^\ast N_{p,q}f||^2\leq k||f||^2. $$
	Thus, $$||\partial\bar{\partial}N_{p,q}f||\leq \sqrt{k}||f|| \mbox {  and  } |(\partial\bar{\partial})^\ast N_{p,q}f||\leq \sqrt{k}||f||. $$
\end{proof}
\begin{remark}\item
\textnormal{With the same assumptions of Thorem \eqref{G}, we find the canonical solution of $\partial\bar{\partial}. $\\ This is because according to the property $(3)$ of the Theorem \eqref{G}, $$f=\partial\bar{\partial}(\partial\bar{\partial})^\ast N_{p,q}f+(\partial\bar{\partial})^\ast\partial\bar{\partial}N_{p,q}f. $$
Using the property $(4)$ of the theorem \eqref{G}, $$N_{p+1,q+1}\partial\bar{\partial}f=\partial\bar{\partial}N_{p,q}f. $$
Since $f\in L^2_{p,q}(\Omega)\cap ker(d)$, then $f$ is $\partial\bar{\partial}$-closed. So $$(\partial\bar{\partial})^\ast\partial\bar{\partial}N_{p,q}f=0. $$
Thus, $$f=\partial\bar{\partial}(\partial\bar{\partial})^\ast N_{p,q}f. $$
So the canonical solution  is given by $u=(\partial\bar{\partial})^\ast N_{p,q}f. $}

\end{remark}
\begin{remark}\item
\textnormal{Remember that $$ \square_{\partial\bar{\partial}}f=\partial\bar{\partial}(\partial\bar{\partial})^\ast f+(\partial\bar{\partial})^\ast\partial\bar{\partial}f. $$
For $p=q=0, $ $(\partial\bar{\partial})^\ast f=0$ for bidegree reasons. In addition  $\partial\bar{\partial}f=0$ because $f\in L^2_{p,q}\cap ker(d). $ \\ So the Laplacian is vanishing.}
\end{remark}
\section{The inverse of  $\tilde{\Delta}_{BC}=(\partial\bar{\partial})(\partial\bar{\partial})^\ast+\partial^\ast\partial+\bar{\partial}^\ast\bar{\partial}$ on $L^2_{p,q}(\Omega). $}
In this section, we focus on the non-elliptic Bott Chern Laplacian (Cf {\color{red}\cite{4}}) denoted $\tilde{\Delta}_{BC}$ and given as $$\tilde{\Delta}_{BC}=(\partial\bar{\partial})(\partial\bar{\partial})^\ast+\partial^\ast\partial+\bar{\partial}^\ast\bar{\partial}. $$
\subsection{Some Lemmas}
Before proving the main Theorem  of this section (Theorem \eqref{E}), we prove some lemmas.
\begin{lemma}\item
	\textnormal{The Laplacian $\tilde{\Delta}_{BC}=(\partial\bar{\partial})(\partial\bar{\partial})^\ast+\partial^\ast\partial+\bar{\partial}^\ast\bar{\partial}$ is closed, densely-defined and self-adjoint.}
\end{lemma}
\begin{proof}
Let's show that $\tilde{\Delta}_{BC}$ is closed:\\
Let $f_n\in Dom(\tilde{\Delta}_{BC})$, 
\begin{align*}
<\Delta_{BC}f_n, f_n>&=<((\partial\bar{\partial})(\partial\bar{\partial})^\ast+\partial^\ast\partial+\bar{\partial}^\ast\bar{\partial})f_n, f_n>\\
&=<(\partial\bar{\partial})(\partial\bar{\partial})^\ast f_n+\partial^\ast\partial f_n+\bar{\partial}^\ast\bar{\partial}f_n, f_n>\\
&=<(\partial\bar{\partial})(\partial\bar{\partial})^\ast f_n, f_n>+<\partial^\ast\partial f_n, f_n>+<\bar{\partial}^\ast\bar{\partial}f_n, f_n>\\
&=<(\partial\bar{\partial})^\ast f_n, (\partial\bar{\partial})^\ast f_n>+<\partial f_n, \partial f_n>+<\bar{\partial}f_n,\bar{\partial} f_n>\\
&=||(\partial\bar{\partial})^\ast f_n||^2+||\partial f_n||^2+||\bar{\partial}f_n||^2
\end{align*}
$\partial$ and $\bar{\partial}$ are the closing operators, let's show that $(\partial\bar{\partial})^\ast$ is a closed operator. \\
indeed we have:\\ $<(\partial\bar{\partial})^\ast f_n, g>=<f_n, \partial\bar{\partial} g>\longrightarrow <f, \partial\bar{\partial} g>=<(\partial\bar{\partial})^\ast f, g>. $\\
Thus, $$<(\partial\bar{\partial})^\ast f_n, g>\longrightarrow <(\partial\bar{\partial})^\ast f, g> \forall\;\;g\in Dom(\partial\bar{\partial})^\ast\cap Dom(\partial\bar{\partial}). $$ 
So $$(\partial\bar{\partial})^\ast f_n\longrightarrow (\partial\bar{\partial})^\ast f. $$
Since $(\partial\bar{\partial})^\ast$ is a closed operator, $$(\partial\bar{\partial})^\ast f_n\longrightarrow(\partial\bar{\partial})^\ast f. $$
Therefore $$\partial\bar{\partial}(\partial\bar{\partial})^\ast f_n\longrightarrow \partial\bar{\partial}(\partial\bar{\partial})^\ast f. $$
Since $\partial$ is a closed operator, 
$$\partial f_n\longrightarrow \partial f. $$
And $\partial^\ast$ is a closed operator, we have:
$$\partial^\ast\partial f_n\longrightarrow \partial^\ast\partial f. $$
Since $\bar{\partial}$ is the closed operator, we have $$\bar{\partial}f_n\longrightarrow \bar{\partial}f. $$
And $\bar{\partial}^\ast$ is also the closed operator, we have:
$$\bar{\partial}^\ast\bar{\partial}f_n\longrightarrow \bar{\partial}^\ast\bar{\partial}f. $$

Than\\ $\partial\bar{\partial}(\partial\bar{\partial})^\ast f_n+\partial^\ast\partial f_n+\bar{\partial}^\ast\bar{\partial}f_n\longrightarrow   \partial\bar{\partial}(\partial\bar{\partial})^\ast f+\partial^\ast\partial f+\bar{\partial}^\ast\bar{\partial}f. $\\ Therefore $$\tilde{\Delta}_{BC}f_n\longrightarrow \tilde{\Delta}_{BC}f. $$
So $\tilde{\Delta}_{BC}$ is a closed operator.\\
\\
Let's now show that $\tilde{\Delta}_{BC}$ is self-adjoint.\\
Let $u,\;\;v\in Dom(\tilde{\Delta}_{BC})$, do we have $<\tilde{\Delta}_{BC}u, v>=<u,\tilde{\Delta}_{BC}v>?$
\begin{align*}
<\tilde{\Delta}_{BC}u, v>&=<((\partial\bar{\partial})(\partial\bar{\partial})^\ast+\partial^\ast\partial+\bar{\partial}^\ast\bar{\partial})u, v>\\
&=<(\partial\bar{\partial})(\partial\bar{\partial})^\ast u+\partial^\ast\partial u+\bar{\partial}^\ast\bar{\partial}u, v>\\
&=<(\partial\bar{\partial})(\partial\bar{\partial})^\ast u, v>+<\partial^\ast\partial u, v>+<\bar{\partial}^\ast\bar{\partial}u, v>\\
&=<(\partial\bar{\partial})^\ast u, (\partial\bar{\partial})^\ast v>+<\partial u, \partial v>+<\bar{\partial}u,\bar{\partial} v>\\
&=< u,\partial\bar{\partial} (\partial\bar{\partial})^\ast v>+< u, \partial^\ast\partial v>+<u,\bar{\partial}^\ast\bar{\partial} v>\\
&=<u, (\partial\bar{\partial} (\partial\bar{\partial})^\ast +\partial^\ast\partial+\bar{\partial}^\ast\bar{\partial})v>\\
&=<u,\tilde{\Delta}_{BC}v>.
\end{align*}
So $\tilde{\Delta}_{BC}$ is self-adjoint.\\
Let's show that  $\tilde{\Delta}_{BC}$ is densely-defined.\\
Let $D^{p,q}(\Omega)$ the space of $(p,q)$ forms with compact support.\\
We have the following inclusions:
$$D^{p,q}(\Omega)\subset Dom(\tilde{\Delta}_{BC})\subset L^2_{p,q}(\Omega). $$
So we have $$\overline{D^{p,q}(\Omega)}\subset\overline{Dom(\tilde{\Delta}_{BC})}\subset\overline{L^2_{p,q}(\Omega)}. $$
Since $L^2_{p,q}(\Omega)$ is a closed, so $L^2_{p,q}(\Omega)=\overline{L^2_{p,q}(\Omega)}$ and $D^{p,q}(\Omega)$ is dense in $L^2_{p,q}(\Omega)$ so $\overline{D^{p,q}(\Omega)}=L^2_{p,q}(\Omega). $
Thus we have: $$L^2_{p,q}(\Omega)\subset\overline{Dom(\tilde{\Delta}_{BC}})\subset L^2_{p,q}(\Omega). $$
Hence $$\overline{Dom(\tilde{\Delta}_{BC}})=L^2_{p,q}(\Omega). $$
So $\tilde{\Delta}_{BC}$ is densely-defined.
\end{proof}
\begin{remark}\item
\textnormal{According to the Hodge decomposition in Hilbert spaces, we have: $$L^2_{p,q}(\Omega)=\overline{R(\tilde{\Delta}_{BC})}\oplus ker(\tilde{\Delta}_{BC}). $$}
\end{remark}

\begin{lemme}\item\label{OL}
\textnormal{Let $\Omega$ be a strictly pseudoconvex bounded starred domain of $\mathbb{C}^n. $ $R(\tilde{\Delta}_{BC})$ is closed and $ker(\tilde{\Delta}_{BC})=\lbrace 0 \rbrace. $}
\end{lemme}
\begin{proof}
We first prove that $ker(\tilde{\Delta}_{BC})=\lbrace 0 \rbrace. $\\
We have: $$ker(\tilde{\Delta}_{BC})=ker(\partial)\cap ker(\bar{\partial})\cap ker((\partial\bar{\partial})^\ast)$$ (Cf {\color{red}\cite{4}}. )\\
Since $\Omega$ is pseudoconvex we have from {\color{red}\cite{2}} $$ker(\partial)\cap ker(\bar{\partial})=\lbrace 0 \rbrace. $$
Thus: $$\lbrace 0 \rbrace\cap ker((\partial\bar{\partial})^\ast)= \lbrace 0 \rbrace .$$
Now let's show that $R(\tilde{\Delta}_{BC})$ is closed\\
We have $$L^2_{p,q}(\Omega)=\overline{R(\tilde{\Delta}_{BC})}\oplus ker(\tilde{\Delta}_{BC}). $$
Or $ker(\Delta_{BC})=\lbrace 0 \rbrace. $\\
And since $R(\tilde{\Delta}_{BC})$ is given by the following relationship from {\color{red}\cite{4}}, we have:
$$R(\tilde{\Delta}_{BC})=R(\partial\bar{\partial})\oplus(R(\partial^\ast)+R(\bar{\partial}^\ast)). $$
Let $f\in L^2_{p,q}(\Omega)$ such that $df=0. $\\
Using the same techniques as for the proposition $2$ of  {\color{red}\cite{3}}, there exists $u\in L^2_{p-1,q-1}(\Omega)$ such que $\partial\bar{\partial}u=f$ et $||u||\leq k||f||. $\\
So $R(\partial\bar{\partial})$ is closed by the lemma $4.1.1$ of  {\color{red}\cite{1}}.\\
According to the Theorem $4.3.4$ of  {\color{red}\cite{2}}, $R(\bar{\partial})$ is closed and $R(\bar{\partial})=ker(\bar{\partial}). $\\
For the same reasons, $R(\partial)$ is closed et $R(\partial)=ker(\partial). $\\
According to the Lemma $4.1.1$ of {\red\cite{1}},
$R(\partial^\ast)$ and $R(\bar{\partial}^\ast)$ are closed.\\
So $R(\tilde{\Delta}_{BC})$ is closed.
\end{proof}

\subsection{Main theorem}
We can now prove the main theorem of this section.
\begin{thm} \label{H}\item
\textnormal{Let $\Omega$ be a strictly  pseudoconvex bounded starred domain of  $\mathbb{C}^n$, $n\geq 2. $ For $1\leq p\leq n$ and $1\leq q\leq n, $ there exists an operator  $$\tilde{N}_{BC}:  L_{p,q}^{2}(\Omega)\longrightarrow L_{p,q}^{2}(\Omega)$$ verifying the following properties:}
\begin{itemize}\item[(1)]
$R(\tilde{N}_{BC})\subset Dom(\tilde{\Delta}_{BC}), $
\item[(2)]
$\tilde{N}_{BC}\tilde{\Delta}_{BC}=\tilde{\Delta}_{BC}\tilde{N}_{BC}=Id$ on $Dom(\tilde{\Delta}_{BC})$
\item[(3)]
For all $f\in L^2_{p,q}(\Omega),$ $f=\partial\bar{\partial}\tilde{N}_{BC}f\oplus((\partial^\ast)+(\bar{\partial}^\ast))\tilde{N}_{BC}f. $
\item[(4)]
For $f\in L^2_{p,q}(\Omega), $
\textnormal{there exists a constant $k\geq 0$ such que:}
\end{itemize}
\textnormal{
$$||\tilde{N}_{BC}f||\leq k||f|| \mbox  {  and  } ||\partial \tilde{N}_{BC}f||\leq \sqrt{k}||f||. $$
$$ ||\bar{\partial} \tilde{N}_{BC}f||\leq \sqrt{k}||f|| \mbox  {  and  }||(\partial\bar{\partial})^\ast \tilde{N}_{BC}f||\leq \sqrt{k}||f||. $$}
\end{thm}
\begin{proof}
$$\tilde{\Delta}_{BC}:Dom(\tilde{\Delta}_{BC})\subset L^2_{p,q}(\Omega)\longrightarrow L^2_{p,q}(\Omega). $$
By definition, the operator $\tilde{N}_{ BC}$ is the inverse of the operator $\tilde{\Delta}_{BC}$ with $$\tilde{N}_{BC}: L^2_{p,q}(\Omega)\longrightarrow Dom(\tilde{\Delta}_{BC})\subset L^2_{p,q}(\Omega). $$
Therefore the properties  $(1)$ and $(2)$  are verified i.e $R(\tilde{N}_{BC})\subset Dom(\tilde{\Delta}_{BC})$ and $\tilde{N}_{BC}\tilde{\Delta}_{BC}=\tilde{\Delta}_{BC}\tilde{N}_{BC}=Id$ on $Dom(\tilde{\Delta}_{BC}). $\\
Let's show the property $(3)$ i.e for all $f\in L^2_{p,q}(\Omega),$ $f=\partial\bar{\partial}\tilde{N}_{BC}f\oplus((\partial^\ast)+(\bar{\partial}^\ast))\tilde{N}_{BC}f. $\\
Since $ker(\tilde{\Delta}_{BC})=\lbrace 0\rbrace $ and $R(\tilde{\Delta}_{BC})$ is closed, the Hodge decomposition gives us the following relation:
\begin{align*}
L^2_{p,q}(\Omega)&=R(\tilde{\Delta}_{BC})\\
&=\partial\bar{\partial}Dom(\tilde{\Delta}_{BC})\oplus((\partial^\ast)+(\bar{\partial}^\ast))Dom(\tilde{\Delta}_{BC}). 
\end{align*}
Therefore $\forall\;\;f\in L^2_{p,q}(\Omega)$, we have:
$$f=\partial\bar{\partial}\tilde{N}_{BC}f\oplus((\partial^\ast)+(\bar{\partial}^\ast))\tilde{N}_{BC}f. $$
Now let's show the property $4$ i.e $$||\tilde{N}_{BC}f||\leq k||f|| \mbox  {  and  } ||\partial \tilde{N}_{BC}f||\leq \sqrt{k}||f||. $$
$$ ||\bar{\partial} \tilde{N}_{BC}f||\leq \sqrt{k}||f|| \mbox  {  and  }||(\partial\bar{\partial})^\ast \tilde{N}_{BC}f||\leq \sqrt{k}||f||. $$
Since $R(\tilde{\Delta}_{BC})$ is closed we have: 
\begin{equation}\label{mum}
||f||\leq k||\tilde{\Delta}_{BC}f||. 
\end{equation}
Since $$R(\tilde{N}_{BC})\subset Dom(\tilde{\Delta}_{BC}), $$ then $$||\tilde{N}_{BC}f||\leq k||\tilde{\Delta}_{BC}\tilde{N}_{BC}f||. $$
So $$||\tilde{N}_{BC}f||\leq k||f||. $$
On the other hand,
\begin{align*}
A&=||(\partial\bar{\partial})^\ast\tilde{N}_{BC}f||^2+||\partial \tilde{N}_{BC}f||^2+||\bar{\partial}\tilde{N}_{BC}f||^2\\
&=<(\partial\bar{\partial})^\ast \tilde{N}_{BC}f,(\partial\bar{\partial})^\ast \tilde{N}_{BC}f>+<\partial \tilde{N}_{BC}f,\partial \tilde{N}_{BC}f>+<\bar{\partial}\tilde{N}_{BC}f,\bar{\partial}\tilde{N}_{BC}f>\\
&=< \tilde{N}_{BC}f,\partial\bar{\partial}(\partial\bar{\partial})^\ast \tilde{N}_{BC}f>+< \tilde{N}_{BC}f,\partial^\ast\partial \tilde{N}_{BC}f>+<\tilde{N}_{BC}f,\bar{\partial}^\ast\bar{\partial}\tilde{N}_{BC}f>\\
&= < \tilde{N}_{BC}f,(\partial\bar{\partial}(\partial\bar{\partial})^\ast +\partial^\ast\partial +\bar{\partial}^\ast\bar{\partial})\tilde{N}_{BC}f>\\
&=<\tilde{N}_{BC}f, f>
\end{align*}
So \\
$||(\partial\bar{\partial})^\ast \tilde{N}_{BC}f||^2+||\partial \tilde{N}_{BC}f||^2+||\bar{\partial}\tilde{N}_{BC}f||^2=<\tilde{N}_{BC}f, f>. $\\
Using  the Cauchy-Schawtz inequality, we have:
$$< \tilde{N}_{BC}f, f>\leq ||\tilde{N}_{BC}f|| ||f||. $$ Since $$||\tilde{N}_{BC}f||\leq k||f||. $$

So $$||(\partial\bar{\partial})^\ast \tilde{N}_{BC}f||^2+||\partial \tilde{N}_{BC}f||^2+||\bar{\partial}\tilde{N}_{BC}f||^2\leq k||f||^2. $$
Therefore
$$||(\partial\bar{\partial})^\ast \tilde{N}_{BC}f||^2\leq k||f||^2\mbox  {  and  } ||\partial \tilde{N}_{BC}f||^2\leq k||f||^2 \mbox  {  and  }||\bar{\partial}\tilde{N}_{BC}f||^2\leq k||f||^2. $$
So
$$||(\partial\bar{\partial})^\ast\tilde{N}_{BC}f||\leq \sqrt{ k}||f|| \mbox  {  and } ||\partial \tilde{N}_{BC}f||\leq \sqrt{ k}||f|| \mbox  {  and  }||\bar{\partial}\tilde{N}_{BC}f||\leq \sqrt{ k}||f||. $$
\end{proof}
\begin{remark}\item\label{II}
\textnormal{In our case, the Laplacian does not verify the following properties: }
\begin{enumerate}
\item[]
$\partial \tilde{N}_{p,q}= \tilde{N}_{p+1, q}\partial $ on $Dom(\partial)$\\
$\bar{\partial} \tilde{N}_{p,q}= \tilde{N}_{p,q+1}\bar{\partial}$ on $Dom(\bar{\partial})$\\
$\partial\bar{\partial}\tilde{N}_{p,q}= \tilde{N}_{p+1, q+1}\partial\bar{\partial}$ on $Dom(\partial\bar{\partial}). $\\
\end{enumerate}
\textnormal{This prevents us from obtaining the canonical solution.}
 \end{remark} 
\section{The inverse of  $\Delta_{BC}=\partial^\ast\partial+\bar{\partial}^\ast\bar{\partial}+(\partial\bar{\partial})(\partial\bar{\partial})^\ast+(\partial\bar{\partial})^\ast(\partial\bar{\partial})+(\partial^\ast\bar{\partial})^\ast(\partial^\ast\bar{\partial})+(\partial^\ast\bar{\partial})(\partial^\ast\bar{\partial})^\ast$ on $L^2_{p,q}(\Omega). $}
 Because of non-ellipticity, terms have been added to the laplacian $\tilde{\Delta}_{BC}$ to make it elliptic (Cf {\color{red}\cite{4}}).
In this section, we work with the elliptic Bott Chern Laplacian given by $$\Delta_{BC}=\partial^\ast\partial+\bar{\partial}^\ast\bar{\partial}+(\partial\bar{\partial})(\partial\bar{\partial})^\ast+(\partial\bar{\partial})^\ast(\partial\bar{\partial})+(\partial^\ast\bar{\partial})^\ast(\partial^\ast\bar{\partial})+(\partial^\ast\bar{\partial})(\partial^\ast\bar{\partial})^\ast. $$ 
\begin{remark}\item
\textnormal{The non-elliptic Bott Chern Laplacian given by $$\tilde{\Delta}_{BC}=(\partial\bar{\partial})(\partial\bar{\partial})^\ast+\partial^\ast\partial+\bar{\partial}^\ast\bar{\partial}$$ and the elliptic one given by $$\Delta_{BC}=\partial^\ast\partial+\bar{\partial}^\ast\bar{\partial}+(\partial\bar{\partial})(\partial\bar{\partial})^\ast+(\partial\bar{\partial})^\ast(\partial\bar{\partial})+(\partial^\ast\bar{\partial})^\ast(\partial^\ast\bar{\partial})+(\partial^\ast\bar{\partial})(\partial^\ast\bar{\partial})^\ast$$ have the same kernel(Cf {\color{red}\cite{4}}). }
\end{remark}
\subsection{Somes Lemmas}
Before proving the main theorem in this section, let's prove some  useful lemmas.
\begin{lemma}\item
\textnormal{The operator $\Delta_{BC}=\partial^\ast\partial+\bar{\partial}^\ast\bar{\partial}+(\partial\bar{\partial})(\partial\bar{\partial})^\ast+(\partial\bar{\partial})^\ast(\partial\bar{\partial})+(\partial^\ast\bar{\partial})^\ast(\partial^\ast\bar{\partial})+(\partial^\ast\bar{\partial})(\partial^\ast\bar{\partial})^\ast$ defined by $L^2_{p,q}(\Omega)\rightarrow L^2_{p,q}(\Omega)$ is closed, densely defined and self-adjoint.}
\end{lemma}
\begin{proof}
Montrons que $\Delta_{BC}$ is closed:\\
Soit $f_n\in Dom(\Delta_{BC})$, 
\begin{align*}
<\Delta_{BC}f_n, f_n>&=<(\partial^\ast\partial+\bar{\partial}^\ast\bar{\partial}+(\partial\bar{\partial})(\partial\bar{\partial})^\ast+(\partial\bar{\partial})^\ast(\partial\bar{\partial})+(\partial^\ast\bar{\partial})^\ast(\partial^\ast\bar{\partial})\\&+(\partial^\ast\bar{\partial})(\partial^\ast\bar{\partial})^\ast)f_n, f_n>\\
&=<\partial^\ast\partial f_n+\bar{\partial}^\ast\bar{\partial}f_n+(\partial\bar{\partial})(\partial\bar{\partial})^\ast f_n+(\partial\bar{\partial})^\ast(\partial\bar{\partial})f_n+(\partial^\ast\bar{\partial})^\ast(\partial^\ast\bar{\partial})f_n\\&+(\partial^\ast\bar{\partial})(\partial^\ast\bar{\partial})^\ast f_n, f_n>\\
&=<\partial^\ast\partial f_n, f_n>+<\bar{\partial}^\ast\bar{\partial}f_n, f_n>+<(\partial\bar{\partial})(\partial\bar{\partial})^\ast f_n, f_n>\\&+<(\partial\bar{\partial})^\ast(\partial\bar{\partial})f_n, f_n>+<(\partial^\ast\bar{\partial})^\ast(\partial^\ast\bar{\partial})f_n, f_n>+<(\partial^\ast\bar{\partial})(\partial^\ast\bar{\partial})^\ast f_n, f_n>\\
&=<\partial f_n, \partial f_n>+<\bar{\partial}f_n,\bar{\partial} f_n>+<(\partial\bar{\partial})^\ast f_n, (\partial\bar{\partial})^\ast f_n>\\&+<(\partial\bar{\partial})f_n, (\partial\bar{\partial})f_n>+<(\partial^\ast\bar{\partial})f_n, (\partial^\ast\bar{\partial})f_n>+<(\partial^\ast\bar{\partial})^\ast f_n, (\partial^\ast\bar{\partial})^\ast f_n>\\
&=||\partial f_n||^2+||\bar{\partial}f_n||^2+||(\partial\bar{\partial})^\ast f_n||^2+||\partial\bar{\partial}f_n||^2+||\partial^\ast\bar{\partial}f_n||^2+||(\partial^\ast\bar{\partial})^\ast f_n||^2
\end{align*}
$\partial$ et $\bar{\partial}$ are closed operators, let's show that $(\partial\bar{\partial})^\ast$ and $\partial\bar{\partial}$ are closed operators as are $\partial^\ast\bar{\partial}$ and $(\partial^\ast\bar{\partial})^\ast. $\\
In fact we have:\\ $<\partial\bar{\partial}f_n, g>=<f_n, (\partial\bar{\partial})^\ast g>\longrightarrow <f, (\partial\bar{\partial})^\ast g>=<\partial\bar{\partial}f, g>. $\\
Thus, $$<\partial\bar{\partial}f_n, g>\longrightarrow <\partial\bar{\partial}f, g> \forall\;\;g\in Dom(\partial\bar{\partial})^\ast\cap Dom(\partial\bar{\partial}). $$ 
So $$\partial\bar{\partial}f_n\longrightarrow \partial\bar{\partial}f. $$
In the same way, we show that $(\partial\bar{\partial})^\ast$ is a closed operator. \\
Let's show that $\partial^\ast\bar{\partial}$ is a closed operator.\\
Let $g\in Dom(\partial^\ast\bar{\partial})^\ast\cap Dom(\partial^\ast\bar{\partial})$, we have:\\
$<\partial^\ast\bar{\partial}f_n,g>=<f_n, \bar{\partial}^\ast\partial g>\longrightarrow <f, \bar{\partial}^\ast\partial g>=<\partial^\ast\bar{\partial}f,g>. $\\
In the same way, we show that $(\partial^\ast\bar{\partial})^\ast$ is a closed operator. \\
Since $\partial$ is a closed operator, 
$$\partial f_n\longrightarrow \partial f. $$
And  $\partial^\ast$ is also a closed operator, we have:
$$\partial^\ast\partial f_n\longrightarrow \partial^\ast\partial f. $$
Since $\bar{\partial}$ is a closed operator, we have $$\bar{\partial}f_n\longrightarrow \bar{\partial}f. $$
And $\bar{\partial}^\ast$ is also a closed operator, we have:
$$\bar{\partial}^\ast\bar{\partial}f_n\longrightarrow \bar{\partial}^\ast\bar{\partial}f. $$
Since $\partial\bar{\partial}^\ast$ is a closed operator, $$\partial\bar{\partial}^\ast f_n\longrightarrow\partial\bar{\partial}^\ast f. $$
Therefore $$\partial\bar{\partial}(\partial\bar{\partial})^\ast f_n\longrightarrow \partial\bar{\partial}(\partial\bar{\partial})^\ast f. $$
In the same way $$\partial\bar{\partial}f_n\longrightarrow \partial\bar{\partial}f. $$
So $$(\partial\bar{\partial})^\ast\partial\bar{\partial}f_n\longrightarrow (\partial\bar{\partial})^\ast\partial\bar{\partial}f. $$
For the same reasons: $$\partial^\ast\bar{\partial}f_n\longrightarrow \partial^\ast\bar{\partial}f. $$
Thus, $$(\partial^\ast\bar{\partial})^\ast\partial^\ast\bar{\partial} f_n\longrightarrow (\partial^\ast\bar{\partial})^\ast\partial^\ast\bar{\partial} f. $$
In the same way $$(\partial^\ast\bar{\partial})^\ast f_n\longrightarrow (\partial^\ast\bar{\partial})^\ast f. $$
Thus, $$\partial^\ast\bar{\partial}(\partial^\ast\bar{\partial})^\ast f_n\longrightarrow \partial^\ast\bar{\partial}(\partial^\ast\bar{\partial})^\ast f. $$
So $$\partial^\ast\partial f_n+\bar{\partial}^\ast\bar{\partial}f_n+\partial\bar{\partial}(\partial\bar{\partial})^\ast f_n+(\partial\bar{\partial})^\ast\partial\bar{\partial}f_n+(\partial^\ast\bar{\partial})^\ast\partial^\ast\bar{\partial} f_n+\partial^\ast\bar{\partial}(\partial^\ast\bar{\partial})^\ast f_n\longrightarrow $$ $$ \partial^\ast\partial f+\bar{\partial}^\ast\bar{\partial}f+\partial\bar{\partial}(\partial\bar{\partial})^\ast f+(\partial\bar{\partial})^\ast\partial\bar{\partial}f+(\partial^\ast\bar{\partial})^\ast\partial^\ast\bar{\partial} f+\partial^\ast\bar{\partial}(\partial^\ast\bar{\partial})^\ast f. $$
Therefore $$\Delta_{BC}f_n\longrightarrow \Delta_{BC}f. $$
So $\Delta_{BC}$ is a closed operator.\\
\\
Now let's show that $\Delta_{BC}$ is self-adjoint.\\
If $u,\;\;v\in Dom(\Delta_{BC})$, do we have  $<\Delta_{BC}u, v>=<u,\Delta_{BC}v>?$
\begin{align*}
<\Delta_{BC}u, v>&=<(\partial^\ast\partial+\bar{\partial}^\ast\bar{\partial}+(\partial\bar{\partial})(\partial\bar{\partial})^\ast+(\partial\bar{\partial})^\ast(\partial\bar{\partial})+(\partial^\ast\bar{\partial})^\ast(\partial^\ast\bar{\partial})\\&+(\partial^\ast\bar{\partial})(\partial^\ast\bar{\partial})^\ast)u, v>\\
&=<\partial^\ast\partial u+\bar{\partial}^\ast\bar{\partial}u+(\partial\bar{\partial})(\partial\bar{\partial})^\ast u+(\partial\bar{\partial})^\ast(\partial\bar{\partial})u+(\partial^\ast\bar{\partial})^\ast(\partial^\ast\bar{\partial})u\\&+(\partial^\ast\bar{\partial})(\partial^\ast\bar{\partial})^\ast u, v>\\
&=<\partial^\ast\partial u, v>+<\bar{\partial}^\ast\bar{\partial}u, v>+<(\partial\bar{\partial})(\partial\bar{\partial})^\ast u, v>\\&+<(\partial\bar{\partial})^\ast(\partial\bar{\partial})u, v>+<(\partial^\ast\bar{\partial})^\ast(\partial^\ast\bar{\partial})u, v>+<(\partial^\ast\bar{\partial})(\partial^\ast\bar{\partial})^\ast u, v>\\
&=<\partial u, \partial v>+<\bar{\partial}u,\bar{\partial} v>+<(\partial\bar{\partial})^\ast u, (\partial\bar{\partial})^\ast v>\\&+<(\partial\bar{\partial})u, (\partial\bar{\partial})v>+<(\partial^\ast\bar{\partial})u, (\partial^\ast\bar{\partial})v>+<(\partial^\ast\bar{\partial})^\ast u, (\partial^\ast\bar{\partial})^\ast v>\\
&=< u, \partial^\ast\partial v>+<u,\bar{\partial}^\ast\bar{\partial} v>+< u,\partial\bar{\partial} (\partial\bar{\partial})^\ast v>\\&+<u,(\partial\bar{\partial})^\ast (\partial\bar{\partial})v>+<u, (\partial^\ast\bar{\partial})^\ast(\partial^\ast\bar{\partial})v>+<u,(\partial^\ast\bar{\partial}) (\partial^\ast\bar{\partial})^\ast v>\\
&=<u, (\partial^\ast\partial+\bar{\partial}^\ast\bar{\partial}+\partial\bar{\partial} (\partial\bar{\partial})^\ast +(\partial\bar{\partial})^\ast (\partial\bar{\partial})+ (\partial^\ast\bar{\partial})^\ast(\partial^\ast\bar{\partial})+(\partial^\ast\bar{\partial}) (\partial^\ast\bar{\partial})^\ast) v>\\
&=<u,\Delta_{BC}v>.
\end{align*}
So $\Delta_{BC}$ is self-adjoint.\\
Let's show that  $\Delta_{BC}$ is densely-defined.\\
Let $D^{p,q}(\Omega)$ be the space of $(p,q)$  forms with compact support.\\ 
We have the following inclusions:
$$D^{p,q}(\Omega)\subset Dom(\Delta_{BC})\subset L^2_{p,q}(\Omega). $$
And $$\overline{D^{p,q}(\Omega)}\subset\overline{Dom(\Delta_{BC})}\subset\overline{L^2_{p,q}(\Omega)}. $$
Now $L^2_{p,q}(\Omega)$ is a closed, $L^2_{p,q}(\Omega)=\overline{L^2_{p,q}(\Omega)}$ and $D^{p,q}(\Omega)$ is dense in  $L^2_{p,q}(\Omega)$ so \\$\overline{D^{p,q}(\Omega)}=L^2_{p,q}(\Omega). $
Thus we have: $$L^2_{p,q}(\Omega)\subset\overline{Dom(\Delta_{BC}})\subset L^2_{p,q}(\Omega). $$
Hence $$\overline{Dom(\Delta_{BC}})=L^2_{p,q}(\Omega). $$
So $\Delta_{BC}$ is densely-defined.
\end{proof}
\begin{remark}\item
\textnormal{According to the Hodge decomposition in Hilbert spaces, we have: $$L^2_{p,q}(\Omega)=\overline{R(\Delta_{BC})}\oplus ker(\Delta_{BC}). $$}
\end{remark}
\begin{lemme}\item\
\textnormal{Let $\Omega$ be a  strictly pseudoconvex   bounded starred domain of $\mathbb{C}^n. $ $R(\Delta_{BC})$ is closed and $ker(\Delta_{BC})=\lbrace 0 \rbrace. $}
\end{lemme}
\begin{proof}
The approach is identical to Lemma \eqref{OL} because $\Delta_{BC}$ and $\tilde{\Delta}_{BC}$ have the same kernel.
\end{proof}
\subsection{Main theorem}
The result of this part is as follows
\begin{thm} \label{E}\item
\textnormal{Let $\Omega$ be a strictly pseudoconvex   bounded starred domain of $\mathbb{C}^n$, $n\geq 2. $ For $1\leq p\leq n$ and $1\leq q\leq n, $ there is an operator  $$N_{BC}:  L_{p,q}^{2}(\Omega)\longrightarrow L_{p,q}^{2}(\Omega)$$ with the following properties:}
\begin{itemize}\item[(1)]
$R(N_{BC})\subset Dom(\Delta_{BC}), $
\item[(2)]
$N_{BC}\Delta_{BC}=\Delta_{BC}N_{BC}=Id$ on $Dom(\Delta_{BC})$
\item[(3)]
For all $f\in L^2_{p,q}(\Omega),$ $f=\partial\bar{\partial}N_{BC}f\oplus((\partial^\ast)+(\bar{\partial}^\ast))N_{BC}f. $
\item[(4)]
For $f\in L^2_{p,q}(\Omega), $
\textnormal{there is a constant $k\geq 0$ such that:}
\end{itemize}
\textnormal{
$$||N_{BC}f||\leq k||f|| \mbox  {  and }||\partial N_{BC}f||\leq \sqrt{k}||f||. $$
$$||\bar{\partial} N_{BC}f||\leq \sqrt{k}||f|| \mbox  {  and  }||\partial\bar{\partial}N_{BC}f||\leq \sqrt{k}||f||. $$
$$||(\partial\bar{\partial})^\ast N_{BC}f||\leq \sqrt{k}||f|| \mbox  {  and  }||\partial^\ast\bar{\partial}N_{BC}f||\leq \sqrt{k}||f||. $$
$$  ||(\partial^\ast\bar{\partial})^\ast N_{BC}f||\leq \sqrt{k}||f||. $$}
\end{thm}

\begin{proof}
The procedure for proving this Theorem \eqref{E} follows the same techniques as the Theorem \eqref{H}.
\end{proof}
\begin{remark}\item
\textnormal{The same problem raised with the remark \eqref{II} arises with the following properties.}
\begin{enumerate}
\item[]
$\partial N_{p,q}= N_{p+1, q}\partial $ on $Dom(\partial)$\\
$\bar{\partial} N_{p,q}= N_{p,q+1}\bar{\partial}$ on $Dom(\bar{\partial})$\\
$\partial\bar{\partial}N_{p,q}= N_{p+1, q+1}\partial\bar{\partial}$ on $Dom(\partial\bar{\partial}). $
\end{enumerate}
\end{remark}
\begin{remark}\item
\textnormal{This is a general remark for the Bott Chern Laplacian (elliptic and non-elliptic). Indeed, with this Laplacian, we have the bijection which allows the invertibility of this Laplacian and, moreover, the existence of the $\partial\bar{\partial}$-Neumann. However we do not obtain commutativity. This justifies the absence of a canonical solution. } 

\end{remark}

\begin{remark}\item
\textnormal{After obtaining the $\partial\bar{\partial}$-Neumann with the Laplacian associated with the $\partial\bar{\partial}$ noted $\square_{\partial\bar{\partial}}$ and that of Bott Chern noted $\Delta_{BC}, $ it would be interesting to look at the link between these different $\partial\bar{\partial}$-Neumann.}
\end{remark}

\end{document}